\documentclass{article}

\usepackage{amssymb,amsfonts,amsthm}
\usepackage{enumitem}
\usepackage[colorlinks=true, urlcolor=blue, linkcolor=blue, citecolor=blue]{hyperref}

\usepackage{xcolor}
\long\def\ignore#1{}
\long\def\cz#1{}% Mark comments
\long\def\czo#1{}% Original text
\long\def\czn#1{}% Note for me

\long\def\czc#1#2{}
\long\def\czd#1{}
\long\def\czD#1{}
\newtheorem{thm}{Theorem}[section]
\newtheorem{cor}[thm]{Corollary}
\newtheorem{prop}[thm]{Proposition}
\newtheorem{lem}[thm]{Lemma}

\theoremstyle{definition}

\newtheorem{problem}[thm]{Problem}

\newtheorem{case}{Case}
\newtheorem{ccase}{Case}[case]
\newtheorem{cccase}{Case}[ccase]

\newtheorem{claim}{Claim}

\theoremstyle{remark}

\makeatletter
\let\c@equation\c@thm
\makeatother
\begin{document}

\title{The Chv\'{a}tal-Erd\H{o}s Condition for Prism-Hamiltonicity}

\author{%
	M. N. Ellingham\thanks{Supported by Simons Foundation award no. 429625.}%
	\qquad Pouria Salehi Nowbandegani \\
 Department of Mathematics, 1326 Stevenson Center,\\
	Vanderbilt University, Nashville, TN 37240\\
	\texttt{mark.ellingham@vanderbilt.edu}\\
	\qquad
	\texttt{pouria.salehi.nowbandegani@vanderbilt.edu}\\
}

\date{6 December 2018}

\maketitle

\begin{abstract}
 The prism over a graph $G$ is the cartesian product $G \Box K_2$. It is
known that the property of having a Hamiltonian prism
(prism-Hamiltonicity) is stronger than that of having a $2$-walk
(spanning closed walk using every vertex at most twice) and weaker than
that of having a Hamilton path. For a graph $G$, it is known that
$\alpha(G) \leq 2 \kappa(G)$, where $\alpha(G)$ is the independence
number and $\kappa(G)$ is the connectivity, imples existence of a
$2$-walk in $G$, and the bound is sharp. West asked for a bound on
$\alpha (G)$ in terms of $\kappa (G)$ guaranteeing prism-Hamiltonicity.
In this paper we answer this question and prove that $\alpha(G) \leq 2
\kappa(G)$ implies the stronger condition, prism-Hamiltonicity of $G$.
 \end{abstract}

%\tableofcontents

\section{Introduction}

In this paper, we consider only simple, finite, and undirected graphs. Let $G$ be a graph. By $\kappa (G)$ and $\alpha (G)$ we mean the connectivity and independence number of $G$, respectively. The {\it prism} over a graph $G$ is the cartesian product $G \Box K_2$. If $G \Box K_2$ is Hamiltonian, we say that $G$ is {\it prism-Hamiltonian}. A {\it$t$-tree} of $G$ is a spanning tree of $G$ with maximum degree at most $t$. A {\it$t$-walk} of $G$ is a spanning closed walk that visits every vertex at most $t$ times.

Kaiser et al.~\cite{Kaiser} showed that the property of having a Hamiltonian prism is stronger than
that of having a $2$-walk and weaker than that of having a Hamilton path, i.e.,
\begin{center}
Hamilton path $\Rightarrow$ prism-Hamiltonian $\Rightarrow$ 2-walk,
\end{center}
and there are examples in \cite{Kaiser} showing that none of these implications can
be reversed. It is of interest to determine whether or not a graph fits in
between the properties of having a Hamilton path and having a $2$-walk. In
particular, which graphs are prism-Hamiltonian even though they may not have a
Hamilton path?

Chv\'{a}tal and Erd\H{o}s  proved the
following theorem.

 \begin{thm} [Chv\'{a}tal and Erd\H{o}s \cite{ChE}\label{ch-e}]
 Let $G$ be a graph with at least three vertices. If $\alpha(G) \leq
\kappa(G)$, then $G$ is Hamiltonian.
 \end{thm}

Suppose $G$ is a graph with $|V(G)| \ge 2$ and $\alpha(G) \leq \kappa(G)
+ 1$. By adding a new vertex $v$ adjacent to all vertices of $G$, we
construct $G'$ which satisfies the hypothesis of Theorem \ref{ch-e}. 
Hence $G'$ is Hamiltonian, so that $G=G'-v$ has a Hamilton path.
 This also holds if $|V(G)|=1$, giving the following corollary.

 % So we have the following corollary.

 \begin{cor} \label{hamilton path}
 Let $G$ be a graph. If $\alpha(G) \leq \kappa(G) + 1$, then $G$ has a
Hamilton path.
 \end{cor}

Moreover, it is known  that $\alpha(G) \leq 2 \kappa(G) $ implies existence of a $2$-walk for $G$ \cite{Jackson}.

\begin{problem} [West \cite{whp}]
 Given $k$, what is the largest value of $a$ such that if $G$ is a graph with $\kappa(G)=k$ and $\alpha(G) = a$, then the prism over $G$ is Hamiltonian?
\end{problem}

For $a > k$, the complete bipartite graph  $K_{k, a}$ is $k$-connected and has independence number $a$. When $a > 2k$, the prism over $K_{k, a}$ is not Hamiltonian, since deleting the $2k$  vertices of degree $a + 1$ leaves $a$ components. Hence the answer to this problem  is at most $2k$.

The following theorem is our answer to this question.

 \begin{thm} \label{main theorem}
 Let $G$ be a graph with at least two vertices.  If $\alpha(G) \leq 2
\kappa (G)$, then $G$ is prism-Hamiltonian.
 \end{thm}

This theorem shows that the Chv\'{a}tal-Erd\H{o}s condition sufficient for being prism-Hamiltonian is the same as for the weaker property of having a $2$-walk.

Here we list the results that we need in our proofs.
\begin{thm} [Bondy and Lov\'{a}sz \cite{b-l}\label{b-l}]
Let $S$ be a set of $k$ vertices in a $k$-connected graph $G$, where $k  \geq 3$. Then
there exists an even cycle in $G$ through every vertex of $S$.
\end{thm}
\begin{thm} [Jackson and Wormald \cite{Jackson}\label{k-walk, k-tree}]
The existence of a $t$-tree implies the existence of a $t$-walk, and the existence of a $t$-walk implies the existence of a $(t+1)$-tree.
\end{thm}
\begin{thm}[Batagelj and Pisanski \cite{BaPi}\label{tree}]
Let $T$ be a tree with maximum degree $\Delta(T) \geq 2
$. Then $T \Box C_t$ is Hamiltonian if and only if $\Delta(T) \leq t$.
\end{thm}

A {\it spanning cactus} in a graph $G$ is
a spanning connected subgraph of maximum degree $3$ that is the union of vertex-disjoint cycles $C_1, C_2, \dots,C_s$ and vertex-disjoint paths $P_1, P_2, \dots, P_t$ such that the graph has no cycles other than $C_1, C_2, \dots, C_s$. The cactus is said to be {\it even} if all of its cycles are even, that is, if the cactus is a bipartite graph.

\begin{lem} [\v{Cada} et al.~\cite{Cada}] \label{ham-hairy}
If $G$ contains a spanning even cactus, then $G$ is prism-Hamiltonian.
\end{lem}

\section{Proof of Theorem \ref{main theorem}}

Recall that Theorem \ref{main theorem} states that if $G$ is a connected
graph then $\alpha(G) \leq 2 \kappa (G)$ implies prism-Hamiltonicity of
$G$.

Let $P = a_1 a_2 \dots a_n$ be a path with $n$ vertices. By $P[a_i,
a_j]$ and $P(a_i, a_j)$ for $1 \leq i < j \leq n$ we mean the paths $a_i
a_{i+1} \dots a_j$ and $a_{i+1} a_{i+2} \dots a_{j-1}$, respectively.
Similarly, we can define   $P[a_i, a_j)$ and $P(a_i, a_j]$.

\begin{proof}[Proof of Theorem \ref{main theorem}]
 If $\alpha(G) \le \kappa(G)+1$ then, by Corollary \ref{hamilton path},
$G$ has a Hamilton path, and hence is prism-Hamiltonian by Lemma
\ref{ham-hairy}.  So we may assume that $\kappa(G)+2 \le \alpha(G) \le
2\kappa(G)$.  Thus, $\kappa(G) \ge 2$.

 We break the proof into two cases, $\kappa(G)=2$ and $\kappa(G) \ge 3$.
 Somewhat surprisingly, we have to work harder in the first case; in the
second case Bondy and Lov\'{a}sz's Theorem \ref{b-l} does a significant
amount of the work.

 \begin{case}% CA
 Suppose that $\kappa(G) = 2$.
 Since $\kappa(G)+2 = 4 \le \alpha(G) \le 2\kappa(G)=4$, we have
$\alpha(G)=4$.
 By adding two adjacent vertices (a complete
graph on two vertices, $K_2$) to $G$ that are adjacent to all vertices
of $G$, we obtain a new graph, say $G'$. Then $\kappa(G')=\alpha(G')=4$.
Therefore by Theorem \ref{ch-e} $G'$ is Hamiltonian. Removing these two
new vertices implies that $G$ has a Hamilton path or two vertex-disjoint
paths $P_1$ and $P_2$ that cover all vertices of $G$. In the former
case $G$ is prism-Hamiltonian, so we assume the latter case. Let $u_1$ and
$u_2$  be the end vertices of $P_1$  and $v_1$ and let  $v_2$  be the
end vertices of $P_2$.

 \begin{claim}% c1
 Each of $P_1$ and $P_2$ contains more than one vertex; otherwise, $G$ is
prism-Hamiltonian.
 \end{claim}

 \begin{proof}
 Suppose $u_1=u_2=u$. Since $G$ is $2$-connected, there are two edges
from $u$ to $P_2$, say $ub_1$ and $ub_2$. If $b_1$ or $b_2$ belongs to
$\{ v_1, v_2 \}$, then $G$ has a Hamilton path, and hence is
prism-Hamiltonian. Now suppose $b_1$ is the neighbor of $u$ closest to
$v_1$ in $P_2$. Since $G$ is $2$-connected, there exists an edge $xy \in
E(G)$ such that $x \in V(P_2[v_1, b_1))$ and $y \in V(P_2(b_1, v_2])$.
One of the cycles $P_1[y, b_2] \cup b_2ub_1 \cup P_2[b_1, x] \cup xy$,
$P_2[b_1,b_2] \cup b_2 u b_1$ or  $P_2[x, y] \cup xy$ is an even cycle and the
even cycle together with remaining two path segments of $P_2$ form a spanning even
cactus, and hence $G$ is prism-Hamiltonian.
 \end{proof}

Suppose  $u_1 \neq u_2$ and $v_1 \neq v_2$. Since $G$ is $2$-connected,
there are distinct vertices $a_1, a_2 \in V(P_1)$ and $b_1, b_2 \in
V(P_2)$ such that $a_1a_2, b_1b_2 \in E(G)$. We may assume that $u_1, a_1, a_2, u_2$ occur in that order on $P_1$, and $v_1, b_1, b_2, v_2$ occur in that order on $P_2$.

 \begin{claim}\label{pathparity}% c2
 The orders of the paths $P_1[a_1, a_2]$ and $P_2[b_1, b_2]$ have
different parity; otherwise, $G$ is prism-Hamiltonian.
 \end{claim}

\begin{proof}
Suppose the orders of the paths $P_1[a_1, a_2]$ and $P_2[b_1, b_2]$ have same parity. Then $P_1[a_1, a_2] \cup a_2b_2 \cup  P_2[b_2, b_1] \cup b_1a_1$ is an even cycle. This cycle together with remaining path segments of $P_1$ and $P_2$ form a spanning even cactus, i.e., the even cycle together with $P_1 - P_1[a_1, a_2]$ and $P_2 - P_2[b_1, b_2]$. Therefore $G$ is prism-Hamiltonian.
\end{proof}

 \begin{claim}\label{pathchord}% c3
 If $P_2[x, y] \cap P_2[b_1, b_2]$ has at least one edge and $xy \in
E(G) \setminus E(P_2)$ for $x, y \in V(P_2)$, then $P_2[x, y] \cup yx$
is an even cycle; otherwise, $G$ is prism-Hamiltonian.
 \end{claim}

\begin{proof}
Suppose $P_2[x, y] \cup yx$ is an odd cycle.  By Claim \ref{pathparity}, $P_1[a_1, a_2] \cup a_2b_2 \cup P_2[b_2, b_1] \cup b_1a_1$ is an odd cycle. Then combining these two odd cycles form an even cycle which yields to a spanning even cactus. (Same statement holds for edges between two vertices of $P_1$.)
\end{proof}

 \begin{claim}% c4
 The set $\{ u_1, u_2, v_1, v_2 \}$ is an independent set; otherwise, $G$
is prism-Hamiltonian.
 \end{claim}

\begin{proof}
By contradiction suppose $\{ u_1, u_2, v_1, v_2 \}$ is not an independent set. If $u_iv_j \in E(G)$ for $i,j \in \{ 1, 2\}$ then $G$ contains a Hamilton path and hence it is prism-Hamiltonian.

Thus, we may assume that $u_1u_2 \in E(G)$, i.e., $P_1 \cup  u_1u_2$ is
a cycle.
 By Claim \ref{pathchord}, $P_1 \cup  u_1u_2$ is an even cycle. We can
assume that  $b_1$  is the closest neighbor of a vertex of $P_1$ on
$P_2$  to $v_1$. Then, by $2$-connectedness and since $b_1$ is the
closest vertex to $v_1$ adjacent to a vertex of $P_1$, there is an edge
$xy$ such that $x \in V(P_2[v_1, b_1))$ and $y \in V(P_2(b_1, v_2])$.
Then by Claim \ref{pathchord}, $P_2[y, x] \cup xy$ is an even cycle. Therefore $P_1
\cup u_1u_2$ and $P_2[x, y] \cup yx$ are even cycles and together with
the edge $a_1b_1$ and remaining path segments of $P_2$ form a spanning
even cactus. \end{proof}

 \begin{claim}\label{noedgebetween}% c5
 There is no  edge $xy$ with $x \in V(P_1)-\{a_1,a_2\}$ and $y \in
V(P_2)-\{b_1,b_2\}$; otherwise, $G$ is prism-Hamiltonian.
 \end{claim}

\begin{proof}
 Suppose $xy \in E(G)$ for $x \in V(P_1)$ and $y \in
V(P_2)$.  By Claim \ref{pathparity}, $P_1[a_1, a_2] \cup a_2b_2 \cup
P_2[b_2, b_1] \cup b_1a_1$ is an odd cycle. Then one of the cycles
$\hbox{$Z_1=$} P_1 [x, a_1] \cup a_1b_1 \cup P_2[b_1, y] \cup yx$ or
$\hbox{$Z_2=$} P_1 [x, a_2] \cup \hbox{$a_2b_2$} \cup P_2[b_2, y] \cup
yx$ is even and together with remaining path segments of $P_1$ and $P_2$
forms a spanning even cactus.
 \end{proof}

 \begin{claim}\label{noedgebetween2}% c6
 There is no edge $xy$ such that either (i) $x \in
\{u_1, u_2\}$ and $y \in V(P_2) - \{b_1, b_2\}$ or (ii) $x \in V(P_1) -
\{a_1, a_2\}$ and $y \in \{v_1, v_2\}$; otherwise, $G$ is
prism-Hamiltonian.
 \end{claim}

 \begin{proof} Without loss of generality suppose that (i) holds with
$x=u_1$.  If $x \ne a_1$ then the result follows by Claim
\ref{noedgebetween}, so suppose that $x=u_1=a_1$.
 The proof of Claim \ref{noedgebetween} fails when $x=a_1$ because if we
need to construct a spanning even cactus from the cycle $Z_1$ then we
would have to attach two path segments of $P_1$ at $x=a_1$, creating a
degree $4$ vertex, which is not allowed.
 However, since $x=u_1=a_1$ here one of these path segments is trivial (just
the single vertex $u_1$) so this does not create a problem now, and we
may proceed as in the proof of Claim \ref{noedgebetween}.
 \end{proof}

Now we may suppose that $\{ u_1, u_2, v_1, v_2 \}$ is an independent
set. By Claim \ref{pathparity}, the paths $P_1[a_1, a_2]$ and $P_2[b_1,
b_2]$  have different parity. Without loss of generality we can assume
that $P_2[b_1, b_2]$ has an odd number of vertices, and therefore there
is a vertex $x\in V(P_2(b_1, b_2))$. Since $\alpha(G)=4$, and $S = \{
u_1, u_2, v_1, v_2 \}$ is an independent set, $x$ is adjacent to some
vertex in $S$. By Claim \ref{noedgebetween2}, we may assume that $x$ is adjacent to neither
$u_1$ nor $u_2$. Without loss of generality we may assume that $x$ is
adjacent to $v_1$. Then by Claim \ref{pathchord}, the cycle $P_2[x, v_1]
\cup v_1x$ is even.
 If $a_1 = u_1$ then we have a spanning even cactus using
the cycle $P_2[x, v_1] \cup v_1x$ and paths $P_2[x,v_2]$ and $b_1u_1
\cup P_1$, so we may assume that $a_1 \ne u_1$. By $2$-connectedness there
is an edge $yz$ such that $y \in V(P_1[u_1, a_1))$ and $z \in V(P_2)
\cup V(P_1(a_1, u_2])$. So we have the following cases.

 \begin{ccase}% CAA
 If $z \in V(P_1(\hbox{$a_1$}, v_2])$, by Claim \ref{pathchord} we may
assume that $yz \cup P_1[z, y]$ is an even cycle. Then the cycles $v_1x
\cup P_2[x, v_1]$ and $yz \cup P_1[z, y]$ together with the edge
$a_1b_1$ and remaining path segments of $P_1$ and $P_2$ form a spanning
even cactus.
 \end{ccase}

 \begin{ccase}% CAB
 Suppose $z \in V(P_2)$. By Claim \ref{noedgebetween} we can assume
that $z=b_1$ or $z=b_2$ which lead us to the following cases.

 \begin{cccase}
 Suppose $z=b_2$. Then we can assume that the cycle $yb_2a_2 \cup
P_1[a_2, y]$ is even; otherwise, the cycle $P_2[b_1, b_2] \cup b_2y
\cup P_1[y, a_1] \cup a_1b_1$ is even and yields a spanning even cactus.
Therefore the even cycles $yb_2a_2 \cup P_1[a_2, y]$ and $v_1x \cup
P_2[x, v_1]$ together with the edge $a_1b_1$ and remaining path segments
of $P_1$ and $P_2$ form a spanning even cactus.
 \end{cccase}

 \begin{cccase}
 Suppose  $z=b_1$. Then for the same reason as above we can assume that
the cycle $yb_1a_1 \cup P_1[a_1, y]$ is even. Therefore there is a
vertex $c \in V(P_1(y, a_1))$. We can assume that $\hbox{$ca_1$} \in
E(P_1)$. Since $\alpha(G)=4$, and $S = \{ u_1, u_2, v_1, v_2 \}$ is an
independent set, $c$ is adjacent to some vertex in $S$. By Claim
\ref{noedgebetween2} we may assume that $c$ is adjacent to neither $v_1$
nor $v_2$. If $u_2c \in E(G)$ then by Claim \ref{pathchord}, $P_1[c,
u_2] \cup u_2c$ is an even cycle and together with $P_2[v_1, x] \cup
xv_1$ it yields a spanning even cactus. Hence we may assume that $u_1c
\in E(G)$.

If $u_1c \notin E(P_1)$, then we may assume that $P_1[c, u_1] \cup u_1c$
is an odd cycle; otherwise, together with $P_2[v_1, x] \cup xv_1$ it
yields a spanning even cactus. If $P_1[c, u_1] \cup u_1c$ is odd, then
$P_1[c, a_2] \cup a_2b_2 \cup P_2[b_2, b_1] \cup b_1y \cup P_1[y, u_1]
\cup u_1c$ is an even cycle and together with remaining path segments of
$P_1$ and $P_2$ forms a spanning even cactus.  Therefore we may assume
that $u_1c \in E(P_1)$, which implies $y=u_1$. Then $P_2[v_1, x] \cup
xv_1$ together with paths $\hbox{$b_1u_1$} \cup P_1$ and $P_2[x, v_2]$
forms a spanning even cactus.
 \end{cccase}
 \end{ccase}% CAB
 \end{case}% CA

 \begin{case}% CB
 Suppose that $k = \kappa(G) \geq 3$. Let $\alpha = \alpha(G)$ and let
$t = \alpha - k \ge 2$. Let $G'$ be the graph $G$ together with a $K_t$ and
all edges from these new $t$ vertices to $V(G)$. Then
$\alpha(G')\hbox{$= \alpha(G)$} \leq \kappa (G')\hbox{$ =
\kappa(G)+t$}$, hence by Theorem \ref{ch-e} $G'$ is Hamiltonian. By
removing these $t$ new vertices, we can cover all the vertices of $G$ by
$r \leq t$ vertex-disjoint paths, $P_1, P_2, \dots, P_r$. Let $v_1,
\dots, v_r$ be one of the end vertex of each of these $r$ paths. By
Theorem \ref{b-l} there is an even cycle, say $C$, passing through $v_1
, \dots, v_r$. Now we put a direction on each of these $r$ paths
starting from $v_i$, $1 \leq i \leq \hbox{$r$}$. Our goal is attaching
some paths to $C$ to form a spanning even cactus.

 Suppose $C$ intersects $P_i$ at $w^i_1 = v_i$, $w^i_2$, $\dots$,
$w^i_{k_i}$, in that order along $P_i$.  Let $x^i_{k_i}$ be the end of
$P_i$ other than $v_i$, and for $1 \le j \le k_i-1$ let $x^i_j$ be the
vertex immediately before $w^i_{j+1}$ on $P_i$.
 Then we add the paths $P_i[w^i_j, x^i_j]$, $1 \le i \le r$, $1 \le j
\le k_i$, to $C$. This process will form a spanning even cactus. Hence,
$G$ is prism-Hamiltonian. \qedhere% so end proof inside case
 \end{case}% CB
 \end{proof}

\section{Conclusion}

It is known \cite[Theorem 5.3]{Jackson} that $\alpha (G) \leq t
\kappa(G)$ implies Hamiltonicity of $G[K_t]$ (the lexicographic product
of $G$ and $K_t$). As an extension of Theorems \ref{ch-e} and \ref{main
theorem} we can ask whether $\alpha (G) \leq t \kappa(G)$ implies
Hamiltonicity of $G \Box K_t$ when $t \ge 3$.  We can
prove the following slightly weaker result.

\begin{prop} \label{main theorem 2}

Let $G$ be a graph, and $t \ge 3$ an integer. If $\alpha (G) \leq
(t-1) \kappa (G)$ then $G \Box C_t$, and hence $G \Box K_t$, is
Hamiltonian.

\end{prop}

\begin{proof}

We know that $\alpha (G) \leq (t-1) \kappa$ implies existence of  a
${(t-1)}$-walk in $G$. By Theorem \ref{k-walk, k-tree} existence of a
$(t-1)$-walk implies the existence of a $t$-tree and hence, by Theorem
\ref{tree}, Hamiltonicity of $G \Box C_t$.
 \end{proof}

 We assume the reader is familiar with the idea of toughness, introduced
by Chv\'{a}tal \cite{Chv73}, who conjectured that for some fixed $t$
every $t$-tough graph is Hamiltonian.
 For $k \ge 3$ we know that $(1/(k-2))$-tough graphs have a $k$-tree and
hence a $k$-walk \cite{Jackson, Win}, and $4$-tough graphs have a
$2$-walk \cite{EZ}.
 Kaiser et al.~\cite[Conjecture 4]{Kaiser} make the natural conjecture
that for some fixed $t$ all $t$-tough graphs are prism-Hamiltonian, and
show that $t$ must be at least $9/8$.

 While it appears very difficult to show that some constant toughness
implies Hamiltonicity or even prism-Hamiltonicity, Chv\'{a}tal-Erd\H{o}s
conditions combined with some simple observations suffice to show that
$\Omega(\sqrt n)$-tough graphs have these properties.
 As far as we can tell, no one has noted this before.
 Suppose $G$ is a non-complete $n$-vertex $t$-tough graph; let
$\alpha=\alpha(G)$ and $\kappa=\kappa(G)$.
 By \cite[Propositions 1.3 and 1.4]{Chv73},  $\kappa \ge 2t$ and $t \le
(n-\alpha)/\alpha$, or $n/(t+1) \ge \alpha$.  Using these, we obtain the
following.

\begin{prop}
Suppose $t > 0$, $n \ge 3$, and $G$ is a $t$-tough $n$-vertex graph.
 \begin{enumerate}[label=(\roman*)]
 \item If $2t(t+1) \ge n$ (e.g., if $t \ge \sqrt{n/2}$), then $G$ is
Hamiltonian.
 \item If $4t(t+1) \ge n$ (e.g., if $t \ge \sqrt n/2$), then $G$ is
prism-Hamiltonian.
 \end{enumerate}
\end{prop}

\begin{proof}
 We may assume $G$ is non-complete.  If $p \ge 0$ and $2pt(t+1) \ge n$
then $p \kappa \ge 2pt \ge n/(t+1) \ge \alpha$.
 Applying Theorem \ref{ch-e} when $p=1$ and Theorem \ref{main theorem}
when $p=2$ gives the result.
 \end{proof}

\end{document}